\documentclass{article}
\usepackage{amsthm,framed,amssymb}
\usepackage[sumlimits]{amsmath}
\usepackage{amsfonts,enumerate}
\usepackage{color,url}
\usepackage[margin=2cm]{geometry}
\usepackage{graphicx}  

\DeclareMathOperator{\diff}{d\!}
\DeclareMathOperator{\dx}{\diff{x}}%{\diff\!\mathbf{x}}

\DeclareMathOperator{\Diff}{Diff}

\newtheorem{theorem}{Theorem}

\newtheorem{corollary}[theorem]{Corollary}

\newtheorem{remark}[theorem]{Remark}

\newcommand{\pp}[2]{\frac{\partial #1}{\partial #2}} 
\newcommand{\dede}[2]{\frac{\delta #1}{\delta #2}}

\newcommand{\rem}[1]{}

\numberwithin{equation}{section}
\numberwithin{figure}{section}

\makeatother

\title{A variational formulation of vertical slice models}
\author{C.J. Cotter and D.D. Holm}

\begin{document}

\title{A variational formulation of vertical slice models} 
\author{C. J. Cotter$^{1}$ and D. D. Holm$^{2}$}
\addtocounter{footnote}{1}
\footnotetext{Department of Aeronautics, Imperial College London. London SW7 2AZ, UK. 
\texttt{colin.cotter@imperial.ac.uk}
\addtocounter{footnote}{1} }
\footnotetext{Department of Mathematics, Imperial College London. London SW7 2AZ, UK. Partially supported by Royal Society of London Wolfson Award and European Research Council Advanced Grant.
\texttt{d.holm@imperial.ac.uk}
\addtocounter{footnote}{1} }

\date{Proc Roy Soc A, to appear}
\maketitle

\makeatother

\maketitle

%\noindent \textbf{AMS Classification:} 


\begin{abstract}
  \noindent A variational framework is defined for vertical slice
  models with three dimensional velocity depending only on $x$ and
  $z$. The models that result from this framework are Hamiltonian, and
  have a Kelvin-Noether circulation theorem that results in a conserved potential
  vorticity in the slice geometry. These results are demonstrated for
  the incompressible Euler--Boussinesq equations with a constant
  temperature gradient in the $y$-direction (the Eady--Boussinesq
  model), which is an idealised problem used to study the formation
  and subsequent evolution of weather fronts. We then introduce a new
  compressible extension of this model. Unlike the incompressible
  model, the compressible model does not produce solutions that are
  also solutions of the three-dimensional equations, but it does
  reduce to the Eady--Boussinesq model in the low Mach number
  limit. Hence, the new model could be used in asymptotic
  limit error testing for compressible weather models running in a
  vertical slice configuration.
\end{abstract}

\maketitle

\tableofcontents

\section{Introduction}

This paper introduces a variational framework for deriving
geophysical fluid dynamics models in a vertical slice geometry
(\emph{i.e.} the $x$-$z$ plane).  The work is motivated by the
asymptotic limit solutions framework advocated in \cite{Cu2007}, in
which model error in dynamical cores for numerical weather prediction
models can be quantified by comparing limits of numerical solutions
with solutions from semigeostrophic (SG) models. In particular, the SG
solutions of the Eady frontogenesis problem specified in a vertical
slice geometry prove very useful since they can be solved in a
two-dimensional domain, which means that they can be run quickly on a
single workstation. In the incompressible hydrostatic and
nonhydrostatic cases these solutions are equivalent to exact solutions
of the full three dimensional equations. As described in
\cite{Cu2007}, this proves to be a challenging test problem. Using a
Lagrangian numerical discretisation that utilises the optimal
transport formulation, converged numerical integrations of the SG
model indicate an almost periodic cycle in which fronts form, change
shape, and then relax again to a smooth solution.  However, primitive
equation solutions obtained by \cite{GaNaHe1992} are rather
dissipative due to the need for eddy viscosity to stabilise the
numerics, and the periodic behaviour is not observed; this leads to a
loss of predictability after the formation of the front. In
\cite{Cu2007}, it is suggested that greater predictability in this
limit might be possible if the numerical solution exhibits energy and
potential vorticity conservation over long time periods; it is also
suggested that a form of Lagrangian averaging may be required to
obtain accurate predictions of the subsequent front evolution.  Since
energy conservation can be derived from a variational framework and
potential vorticity arises from the particle relabelling symmetry,
this has motivated us to develop such a framework in the case of
``slice geometries'' in which there are three components of velocity,
but they are functions of $x$ and $z$ only. 

Another motivation for our work is that efforts to compare
compressible models with the two dimensional SG solutions have been
thwarted by the fact that it is not possible to construct a
compressible vertical slice model with solutions that are consistent
with the full three dimensional model, with conserved energy and
potential vorticity. This is because of the nonlinear dependence in
the equation of state on the $y$-dependent component of the
temperature. Hence, so far asymptotic limit studies of compressible
models have only been performed over short time intervals
corresponding to the initial stages of front formation \cite{Cu2008}.
In this paper we introduce a new compressible slice model that can be
used in asymptotic limit studies, since it has a conserved energy and
potential vorticity. The price to pay is that the solutions are not
consistent with the full three dimensional equations. However, the model
should still be very useful in studying the behaviour of discretisation
methods and averaging procedures for numerical weather prediction in
the presence of fronts.

Our approach is to derive models in the Euler-Poincar\'e framework \cite{HoMaRa1998}.
This framework is a way of obtaining variational models without
resorting to Lagrangian coordinates, by providing formulas that
express how infinitesimal variations in the Lagrangian flow map
correspond to variations in the Eulerian prognostic variables.  The present 
paper specialises to the case where all the Eulerian fields are
independent of $y$. This corresponds to a subgroup of the group of
diffeomorphisms in three dimensions, which can be expressed as a
semi-direct product of two dimensional diffeomorphisms in the vertical
slice and rigid displacements in the $y$-direction. Having selected
this group, the Euler-Poincar\'e theory immediately tells us how to
perform Hamilton's principle. In this framework, the problem of developing slice models reduces to the problem of choosing which Lagrangian to substitute into the action.

The structure of this paper is as follows.  In Section \ref{slice
  def}, we identify the slice subgroup, and set up the geometric
framework. In Section \ref{hamiltons}, we then obtain the general
equations of motion corresponding to the Euler-Poincar\'e equation
with advected density and tracer variables (temperature).  In Section
\ref{KNthm-sec} we reformulate the equations in a more geometric
notation, and show that the equations conserve energy in the case of
Lagrangians without explicit time-dependence; this is shown by
recasting the equations in Lie-Poisson form. We also show that the
equations have a Kelvin-Noether circulation theorem. This circulation
theorem differs from the usual circulation theorem for baroclinic
fluids which have a baroclinic circulation production term on the
right-hand side that only vanishes if the circulation loop lies on an
isentropic surface. In the slice geometry, this baroclinic term can be
rewritten as the time-derivative of another circulation term, and we
obtain conservation of the total circulation on arbitrary curves
within the slice. This circulation theorem leads to a conserved
potential vorticity that turns out to correspond to the usual
three-dimensional Ertel potential vorticity. We then use this
framework to present a number of models in the slice geometry. In
Section \ref{eady incompressible} we show how to obtain the
Euler-Boussinesq Eady model. We present the corresponding
Lagrangian-averaged Eady model in Section \ref{alpha model} and
introduce our new compressible slice model in Section
\ref{compressible}, comparing it with the model used in
\cite{Cu2008}. Finally we provide a summary and outlook in Section
\ref{summary}. The appendices provide proofs and show how this
framework relates to known Lie-Poisson formulations of superfluid
models. This relationship is significant since it shows how to build
conservative numerical schemes in the slice geometry.
This last point is also discussed in Section \ref{summary}.

\section{Vertical slice models}

\subsection{Definition}
\label{slice def}
Physically, slice models are used to describe the formation of fronts
in the atmosphere and ocean. These fronts arise when there is a strong
North-South temperature gradient (maintained by heating at the Equator
and cooling at the Pole), which maintains a vertical shear flow in the
East-West direction through geostrophic balance. In an idealised
situation, neglecting the Earth's curvature and assuming a constant
Coriolis parameter $f$, this basic steady state can be modelled with a
three-dimensional flow in which there is a constant temperature
gradient in the $y$-direction, and the velocity points in the
$x$-direction with a linear shear in the $z$-direction. This basic
flow is unstable to $y$-independent perturbations in all three
components of velocity and temperature which rapidly lead to the
formation of fronts that vary sharply in the $x$ direction but do not
vary in structure in the $y$ direction. The presence of the constant
gradient of the temperature in the $y$ direction means that the
$y$-component of velocity is coupled to other variables since it can
lead to a source or sink of temperature in each vertical slice. Since
all of the perturbations are $y$-independent, we can consider the
dynamics in a single vertical slice without loss of generality.

To build a variational vertical slice model of this type, it is
assumed that the forward Lagrangian map takes the form
\begin{equation}
\label{eq:slice map}
\phi(X,Y,Z,t) = \left(x(X,Z,t),y(X,Z,t)+Y,z(X,Z,t)\right),
\end{equation}
where $(X,Y,Z)$ are Lagrangian labels, $(x,y,z)$ are particle
locations and $t$ is time, \emph{i.e.}
\[
\frac{\partial\phi}{\partial Y} =
\begin{pmatrix}
0  \\
 1 \\
0  \\
\end{pmatrix}.
\]
Such maps form a subgroup of the diffeomorphisms%
\footnote{Diffeomorphisms are smooth invertible maps with smooth inverses.}
$\Diff(\Omega\times \mathbb{R})$ (where $\Omega\in \mathbb{R}^2$ is the domain in the $x$-$z$ plane, and $\mathbb{R}$
represents an infinite line in the $y$-direction). This subgroup is
isomorphic to $\Diff(\Omega)\circledS \mathcal{F}(\Omega)$ where
$\circledS$ denotes the semidirect product, and $\mathcal{F}(\Omega)$ denotes an
appropriate space of smooth functions on $\Omega$ that specify the
displacement of Lagrangian particles in the $y$-direction at each
point in $\Omega$. Multiplication in the semidirect product group is given by a standard formula
\cite{HoMaRa1998},
\begin{equation}
(\phi_1,f_1)\cdot(\phi_2,f_2) = (\phi_1\circ\phi_2, \phi_1\circ f_2 + f_1).
\label{SDgroupaction}
\end{equation}
The corresponding Lie algebra is isomorphic to
$\mathfrak{X}(\Omega)\circledS \mathcal{F}(\Omega)$ where
$\mathfrak{X}(\Omega)$ denotes the vector fields on $\Omega$,
representing the two components of the velocity $u_S\in\mathfrak{X}(\Omega)$ in the $x$-$z$ plane, and the smooth function 
$u_T\in \mathcal{F}(\Omega)$ represents the $y$-component of the
velocity.  We write elements of
$\mathfrak{X}(\Omega)\circledS\mathcal{F}(\Omega)$ as $(u_S,u_T)$ where
$u_S$ is the ``slice'' component in the $x$-$z$ plane, and $u_T$ is the
``transverse'' component in the $y$ direction. In component notation, the \emph{Lie bracket} for the Lie algebra $\mathfrak{X}(\Omega)\circledS \mathcal{F}(\Omega)$ of the semidirect product group $\Diff(\Omega)\circledS \mathcal{F}(\Omega)$ takes the form
\begin{equation}
[(u_S,u_T),(w_S,w_T)] = \left([u_S,w_S] ,\,u_S\cdot\nabla w_T - w_S\cdot\nabla u_T \right),
\label{SDalgebraaction}
\end{equation}
where $[u_S,w_S]=u_S\cdot\nabla w_S - w_S\cdot\nabla u_S$ is the Lie
bracket for the time-dependent vector fields
$(u_S,w_S)\in\mathfrak{X}(\Omega)$, and $\nabla$ denotes the gradient
in the $x$-$z$ plane.

We introduce two types of advected quantities in this framework. 

First, mass is conserved locally, so the mass element $D\diff ^3\!x$ is advected in three-dimensional space. That is, the mass density $D(x,y,z,t)$ satisfies  
\[
(\partial_t+\mathcal{L}_{(u_S,u_T)})(D\diff^3\!x)
=
\big(\partial_t D + \nabla \cdot ( u_S D) + \partial_y(u_TD)\big)
\diff^3\!x
= 0
\,,
\]
with partial time derivative $\partial_t=\partial/\partial t$ and
partial space derivative $\partial_y=\partial/\partial y$ in the
${y}$-direction normal to the $x$-$z$ plane.
If $u_T$ and $D$ are specified to be $y$-independent consistently with
the slice motion assumption, then the last term vanishes and the
equation for conservation of mass reduces to advection of an
\emph{areal density} $D\diff S\in \Lambda^2(\Omega)$, in which
$D(x,z,t)$ satisfies the continuity equation,
\begin{align}
\partial_tD + \nabla\cdot(u_SD) = 0
\,.
\label{D-eqn}
\end{align}
Second, in order to represent potential temperature that has a
constant gradient in the $y$-direction, ${s}
=\partial\bar{\theta}/\partial y=$ constant, we shall require advected
scalars $\theta(x,y,z,t)$ that may be decomposed into dynamic and
static parts, as
\begin{align}
\theta(x,y,z,t) = \theta_S(x,z,t) + (y-y_0){s} .
\label{theta-decomp}
\end{align}
Consequently, the three-dimensional scalar tracer equation 
\[
\partial_t\theta_S + u_S\cdot\nabla\theta_S + u_T\partial_y\theta =0
\]
becomes a dynamic equation for 
$\theta_S(x,z,t)\in\mathcal{F}(\Omega)$, which satisfies,
\begin{align}
\partial_t\theta_S + u_S\cdot\nabla\theta_S + u_T{s} =0
\,,
\label{theta-eqn}
\end{align}
in which we keep in mind that $s$ is a constant and $u_T(x,z,t)$ has been specified to be  $y$-independent.
The space of advected scalars of this type is isomorphic to 
$\mathcal{F}(\Omega)\times\mathbb{R}$, represented as pairs
$(\theta_S,{s})$, with infinitesimal
Lie algebra action
\[
\mathcal{L}_{(u_S,u_T)}(\theta_S,\,{s} )
=\left(u_S\cdot\nabla\theta_S + u_T{s} ,0\right)
.
\]

\subsection{Variational formulation via Hamilton's principle}
\label{hamiltons}
In this section we show how to perform variational calculus in the
slice geometry. Vector fields of infinitesimal variations $(w_S,w_T)$
in the Lie algebra $\mathfrak{X}(\Omega)\circledS\mathcal{F}(\Omega)$
of the semidirect product group $\Diff(\Omega)\circledS
\mathcal{F}(\Omega)$ induce infinitesimal variations in $(u_S,u_T)$,
$D$, and $(\theta_S,{s})$ as follows:\footnote{These are standard
  formulas for defining the variations in Hamilton's principle. See
  \cite{HoMaRa1998} and Appendix \ref{EPForm} for details.}
\begin{align}
\begin{split}
\delta \left(u_S,u_T\right) & =  \left(\partial_t{w}_S + 
[u_S,w_S],\partial_t{w}_T + u_S\cdot\nabla w_T - w_S\cdot\nabla u_T
\right), \\
\delta D & =  -\nabla\cdot(w_SD)
\,, \\
\delta\left(\theta_S,{s}\right)
 & =  \left(-w_S\cdot\nabla\theta_S - w_T{s},0\right).
\end{split}
\label{var-formula}
\end{align}
For a Lagrangian functional $l[(u_S,u_T),(\theta_S,{s} ),D]: 
(\mathfrak{X}\circledS\mathcal{F}(\Omega))\circledS ( (\mathcal{F}(\Omega) \times \mathbb{R}) \times \Lambda^2(\Omega))\to\mathbb{R} $, we apply Hamilton's principle and obtain
\begin{align}
\begin{split}
0 &= \delta S \\
&= \delta \int_0^T l\big[(u_S,u_T),(\theta_S,{s} ),D\big] \diff{t} 
\\
&=  \int_0^T
\left\langle 
\dede{l}{(u_S,u_T)},\delta (u_S,u_T) 
\right\rangle
+
\left\langle 
\dede{l}{(\theta_S,{s} )},\delta (\theta_S,{s} ) 
\right\rangle
+ \left\langle \dede{l}{D}, \delta D \right\rangle
\diff{t} 
 \\
& =  \int_0^T \left\langle \dede{l}{u_S},\, 
\partial_t{w}_S + (u_S\cdot\nabla)w_S - (w_S\cdot\nabla)u_S
\right\rangle + \left\langle \dede{l}{u_T},\,
\partial_t{w}_T + u_S\cdot\nabla w_T - w_S\cdot\nabla u_T\right\rangle
\\
&  \qquad 
+ \left\langle \dede{l}{D},\, -\nabla\cdot(w_SD)
\right\rangle 
+ \left\langle \dede{l}{\theta_S}\,,\,
-(w_S\cdot\nabla)\theta_S - w_T{s}
\right\rangle \diff{t} \\
& =  \int_0^T \left\langle -\pp{}{t}\dede{l}{u_S} - \nabla\cdot
\left(u_S\otimes \dede{l}{u_S}\right) - (\nabla u_S)^T\dede{l}{u_S}
- \dede{l}{u_T}\nabla u_T+D\nabla\dede{l}{D}
-\dede{l}{\theta_S}\nabla\theta_S
,\,w_S
\right\rangle \\
&  \qquad + \left\langle -\pp{}{t}\dede{l}{u_T} - \nabla\cdot\left(u_s
\dede{l}{u_T}\right) - \dede{l}{\theta_S}{s}
,\,
w_T\right\rangle  \diff{t} \\
&  \qquad +
\left[
\left\langle \dede{l}{u_S},\,w_S \right\rangle
+
\left\langle \dede{l}{u_T},\,w_T \right\rangle
\right]_0^T\!\!,
\end{split}
\label{HP-vectornotation}
\end{align}
where the angle brackets indicate $L_2$ inner products with
integration over $\mathbb{R}^2$. The last term makes no contribution
for velocity variations $(w_S,w_T)$ that vanish at the endpoints in
time.

Hence, we obtain the Euler-Poincar\'e equations on the slice semidirect
product with advected density $D$ and scalar $\theta$:
\begin{align}
\begin{split}
\pp{}{t}\dede{l}{u_S} + \nabla\cdot \left(u_S\otimes
\dede{l}{u_S}\right) + (\nabla u_S)^T {\cdot\, } \dede{l}{u_S} +
\dede{l}{u_T}\nabla u_T &= D\nabla\dede{l}{D}
-\dede{l}{\theta_S}\nabla\theta_S, \\ \pp{}{t}\dede{l}{u_T} +
\nabla\cdot\left(u_{S} \dede{l}{u_T}\right) & =
-\,\dede{l}{\theta_S}{s} \,.
\end{split}
\label{EPSD-system}
\end{align}
  The system (\ref{EPSD-system}) is
completed by including the advection equations (\ref{D-eqn}) and
(\ref{theta-eqn}) for $D$ and $\theta_S$, respectively.

\subsection{Geometric reformulation and Kelvin-Noether circulation
  theorem}\label{KNthm-sec}

\begin{theorem}[Energy conservation]
\label{energythm}
If the Lagrangian $l$ has no explicit time-dependence, the energy
functional defined by the Legendre transformation
\begin{equation}
\label{eq:energy}
h[(m_S,\,m_T),(\theta_S,{s} ),D]
= 
\left\langle (m_S,\,m_T), (u_S,\,u_T) \right\rangle - l[(u_S,\,u_T),(\theta_S,{s} ),D]
\,.
\end{equation}
is conserved for solutions of Equations (\ref{D-eqn}),
(\ref{theta-eqn}) and (\ref{EPSD-system}).
\end{theorem}
\begin{proof}
  In Appendix \ref{HamForm}, we show that Equations (\ref{D-eqn}),
  (\ref{theta-eqn}) and (\ref{EPSD-system}) are Hamiltonian, with
  Hamiltonian given by $h$ in Equation \eqref{eq:energy}. If $l$ has
  no explicit time-dependence, then $h$ has no explicit
  time-dependence and is therefore an invariant of the Hamiltonian
  system.
\end{proof}
\begin{theorem}[Kelvin-Noether circulation theorem]
\label{KNcircthm}
Equations (\ref{D-eqn}), (\ref{theta-eqn}) and (\ref{EPSD-system}) imply a conservation law for circulation,
\begin{align}
\frac{d}{dt}\oint_{c(u_S)}\hspace{-1mm}
\left(s\left(\frac1D\dede{l}{u_S}\right)  - \left(\frac1D\dede{l}{u_T}\right)\,\nabla\theta_S  \right)\cdot \dx
=
0
\,,
\label{EPSD-circthm}
\end{align}
in which $c(u_S)$ is a circuit in the vertical slice moving with velocity $u_S$ and 
$s={s}$ is a constant parameter. 
\end{theorem}

\begin{proof}
The proof of the theorem is facilitated by rewriting the system of equations (\ref{D-eqn}), (\ref{theta-eqn}) and (\ref{EPSD-system}) equivalently in the following geometric form,
\begin{align}
\begin{split}
\left(\pp{}{t} + \mathcal{L}_{u_S}\right)\left(\frac1D\dede{l}{u_S} \cdot \dx\right)
+ \left(\frac1D\dede{l}{u_T}\right)\diff u_T
 &= \diff\left(\dede{l}{D}\right)
-\left(\frac1D\dede{l}{\theta_S}\right)\diff\theta_S
\,, \\
\left(\pp{}{t} + \mathcal{L}_{u_S}\right)\left(\frac1D\dede{l}{u_T}\right)
& =    -\left(\frac1D\dede{l}{\theta_S}\right){s}
\,,\\
\left(\pp{}{t} + \mathcal{L}_{u_S}\right)\theta_S + u_T{s} &= 0
\,,\\
\left(\pp{}{t} + \mathcal{L}_{u_S}\right)(D\diff S)&= 0
\,,
\end{split}
\label{EPSD-geom-form}
\end{align}
where $\mathcal{L}_{u_S}$ denotes \emph{Lie derivative} along the vector field $u_S$. One may then verify the circulation theorem (\ref{EPSD-circthm}) for slice models by applying the relation
\begin{align*}
\frac{d}{dt}\oint_{c(u_S)}\hspace{-1mm} 
v(x,t) \cdot \dx
=
\oint_{c(u_S)}\hspace{-1mm} 
\left(\pp{}{t} + \mathcal{L}_{u_S}\right) \left(v(x,t) \cdot \dx\right) 
,
\end{align*}
for any vector $v(x,t)$ in the slice. 

\begin{corollary}\label{EPSCM-PVthm}
The system of equations (\ref{EPSD-geom-form}) implies that the following potential vorticity (PV,  denoted as $q$) is conserved along flow lines of the fluid velocity $u_S$,
\begin{align}
\partial_t q + u_S\cdot\nabla q = 0
\quad\hbox{for potential vorticity}\quad
q := \frac{1}{D}\left( 
{\rm curl}\,\left(s\frac1D\dede{l}{u_S}\right) 
+ 
 \nabla \theta_S \times \nabla \left(\frac1D\dede{l}{u_T}\right) \right) \cdot \hat{y}
\,.
\label{EPSD-PVthm}
\end{align}
\end{corollary}
\begin{proof}
Applying the differential operation $\diff\,$ to the first equation in  the system (\ref{EPSD-geom-form}) yields
\begin{align}
\left(\partial_t + \mathcal{L}_{u_S}\right)
\left( \left( 
{\rm curl}\,\left(\frac1D\dede{l}{u_S}\right) 
+ 
s^{-1} \nabla \theta_S \times \nabla \left(\frac1D\dede{l}{u_T}\right) \right) \cdot \hat{y} \diff S\right) =
0
\,,
\label{EPSD-circ4}
\end{align}
where $\diff S$ is the surface element in the vertical slice, whose
normal vector is $\hat{y}$.  Applying the Lie derivative and using the
continuity equation for $D$ then yields the local conservation law
(\ref{EPSD-PVthm}).
\end{proof}

Upon introducing the new notation,
\begin{align}
\begin{split}
v_S := \frac1D\dede{l}{u_S}
\,,\quad
v_T := \frac1D\dede{l}{u_T}
\,,\quad
\pi := \,\dede{l}{D}
\,,\quad
\gamma_S := \frac1D\dede{l}{\theta_S}
\,,
\end{split}
\label{EPSD-newnotation}
\end{align}
the system (\ref{EPSD-geom-form}) takes a slightly more 
transparent form
\begin{align}
\begin{split}
\left(\partial_t + \mathcal{L}_{u_S}\right)\left(v_S \cdot \dx\right)
 &= \diff \pi
 - v_T\diff u_T
 - \gamma_S \diff\theta_S
\,, \\
\left(\partial_t + \mathcal{L}_{u_S}\right)v_T
& =     -\,s\,\gamma_S 
\,,\\
\left(\partial_t + \mathcal{L}_{u_S}\right)\diff \theta_S  &= -\,s\diff u_T 
\,,\\
\left(\partial_t+\mathcal{L}_{u_S}\right)(D\diff S)
 &= 0
\,,
\end{split}
\label{EPSD-geom2}
\end{align}
in which the differential of the third equation has also been taken.
Hence, combining the middle two equations in (\ref{EPSD-geom2}) results in
\begin{align}
\left(\partial_t + \mathcal{L}_{u_S}\right)(v_T\diff \theta_S)
= -s (v_T\diff u_T+\gamma_S \diff\theta_S)
\,.\label{EPSD-geom2a}
\end{align}
Inserting this formula into the first equation in (\ref{EPSD-geom2}) implies that 
\begin{align}
\left(\partial_t + \mathcal{L}_{u_S}\right)\left(s v_S \cdot \dx 
-  v_T\diff \theta_S  \right)
= \diff \pi 
\,.
\label{EPSD-geom3}
\end{align}
This relation then
yields the Kelvin-Noether circulation theorem
as stated above in (\ref{EPSD-circthm}),
\begin{align}
\frac{d}{dt}\oint_{c(u_S)}\hspace{-2mm}
\left(sv_S  -  v_T \nabla \theta_S  \right)\cdot \dx
=
\oint_{c(u_S)}\hspace{-2mm}
\left(\partial_t + \mathcal{L}_{u_S}\right)\left(s v_S \cdot \dx -  v_T \diff \theta_S  \right)
=
\oint_{c(u_S)} \hspace{-2mm}
d\pi 
=0
\,,
\label{EPSD-circ3}
\end{align}
and potential vorticity conservation as in (\ref{EPSD-PVthm}),
\begin{align}
\partial_t q + u_S\cdot\nabla q 
= 0
\quad\hbox{for potential vorticity}\quad
q := \frac{1}{D}\left( s\,
{\rm curl}\,v_S 
+ 
 \nabla \theta_S \times \nabla v_T\right) \cdot \hat{y}
\,.
\label{EP-PVcons}
\end{align}

\end{proof}
\begin{remark}
Note that this circulation theorem is different from the case of
general 3D motions, in which the circulation is only preserved if the
loop integral is restricted to lie on a temperature isosurface. In the
special case of slice motions, the baroclinic generation term can
itself be written as the total derivative of a loop integral. The
physical interpretation is that $q$ is in fact the usual three-dimensional
potential vorticity. Due to the existence of the linear $y$-variation in
$\theta$, it is always possible to find an equivalent three-dimensional
loop on a temperature isosurface that projects onto any given two-dimensional
loop in the vertical slice plane.
\end{remark}
\begin{remark}
In Appendix \ref{X(S)X-append} we will discuss the geometric meaning
of the Kelvin-Noether circulation theorem (\ref{EPSD-circthm}) and the
potential vorticity conservation law (\ref{EPSD-PVthm}) from the
viewpoint of the Lie-Poisson brackets in the Hamiltonian formulation
of these equations.
\end{remark}

\section{The Euler--Boussinesq Eady model}
\label{eady incompressible}
\subsection{Specialising the Euler--Poincar\'e equations to deal with
  the Eady model}
The Euler--Boussinesq Eady model in a periodic
channel of width $L$ and height $H$, has Lagrangian
\begin{align}
l[u_S, u_T,D,\theta,p] = 
\int_{\Omega} \frac{D}{2}\left(|u_S|^2 + u_T^2\right)
+D f u_T x
+\frac{g}{\theta_0} D\left(z-\frac{H}{2}\right)\theta_S
+ p(1-D) \diff{V},
\label{Eady-Lag}
\end{align}
where $g$ is the acceleration due to gravity, $\theta_0$ is the
reference temperature, $f$ is the Coriolis parameter, and we have
introduced the Lagrange multiplier $p$ to enforce constant density. We
obtain the following variational derivatives of this Lagrangian,
\begin{align}
\begin{split}
v_S & = \frac1D\dede{l}{u_S}  =  u_S
\,, \qquad
v_T = \frac1D\dede{l}{u_T}  =  u_T + fx
\,, \\
\pi &= \dede{l}{D}  =  \frac{1}{2}\left(|u_S|^2 + u_T^2\right)
+ fu_Tx
- p + \frac{g}{\theta_0}\theta_S\left(z-\frac{H}{2}\right)
, \\
\gamma_S & = \frac1D\dede{l}{\theta_S}  
=  \frac{g}{\theta_0}\left(z-\frac{H}{2}\right)
, \qquad
\dede{l}{p}  =  1-D\,.
\end{split}
\label{Eady-Lag-vars}
\end{align}
Substitution of these variational derivatives into the Euler-Poincar\'e equations in (\ref{EPSD-system}) gives 
\begin{align}
\begin{split}
\partial_tu_S 
&+ 
u_S\cdot\nabla u_S  + (\nabla u_S)^T\cdot u_S + (u_T + fx)\nabla u_T
 \\ &
= \nabla\left(\frac{1}{2}\left(|u_S|^2 + u_T^2\right) + u_T f x 
 - p + \frac{g}{\theta_0}\theta_S\left(z-\frac{H}{2}\right)\right) 
 -\frac{g}{\theta_0}\left(z-\frac{H}{2}\right)\nabla\theta_S, \\
\partial_t u_T &+ u_S\cdot\nabla (u_T+fx)  =  -\frac{g}{\theta_0}\left(z-\frac{H}{2}\right){s}.
\end{split}
\label{Eady-EPSDeqns}
\end{align}
Upon substituting $D=1$, $\nabla\cdot{u}_S=0$ and combining with equations (\ref{D-eqn}) and (\ref{theta-eqn}) for the advected quantities $D$ and $\theta$, the system of equations (\ref{Eady-EPSDeqns}) becomes
\begin{align}
\begin{split}
\partial_t u_S + u_S\cdot\nabla u_S 
-fu_T{\hat{x}}
&= -\nabla p 
 + \frac{g}{\theta_0}\theta_S\hat{z}, \\
\partial_t u_T + u_S\cdot\nabla u_T 
+ fu_S\cdot\hat{{x}}
& =   -\frac{g}{\theta_0}\left(z-\frac{H}{2}\right){s}, \\
\nabla\cdot u_S & =  0, \\
\partial_t \theta_S + u_S\cdot\nabla\theta_S + u_T{s} & =  0,
\end{split}
\label{Eady-EPSDeqns1}
\end{align}
where $\hat{{x}}$ is the unit normal in the x-direction.  \\

\begin{remark}
The system (\ref{Eady-EPSDeqns1}) is the \emph{standard} Euler-Boussinesq Eady slice model.
\end{remark}

\subsection{Geometric reformulation and circulation theorem for the Eady model}

Substitution of the variational derivatives in (\ref{Eady-Lag-vars})  into the geometric form of the system of Euler-Poincar\'e equations in (\ref{EPSD-geom2}) gives the following equivalent form of this system,
\begin{align}
\begin{split}
\left(\partial_t + \mathcal{L}_{u_S}\right)\left(u_S \cdot \dx\right)
 &= -\diff p - (u_T+fx)d u_T
 -  \theta_S \diff\gamma_S
\,, \\
\left(\partial_t + \mathcal{L}_{u_S}\right)(u_T+fx)
& =   -\,s\,\gamma_S 
\,,\\
D=1 
\quad\Longrightarrow\quad
\nabla \cdot {u_S} &= 0
\,,\\
\left(\partial_t + \mathcal{L}_{u_S}\right)\diff \theta_S  &= -\,s\diff u_T 
\,.
\end{split}
\end{align}

Consequently, we recover the Kelvin circulation conservation law (\ref{EPSD-circ3}) for the Eady model in the form
\begin{align}
\frac{d}{dt}\oint_{c(u_S)}\hspace{-2mm}
\left(su_S  - (u_T+fx)\nabla \theta_S  \right)\cdot \dx
=
\oint_{c(u_S)} \hspace{-2mm}
\diff \left(\frac12 |u_S|^2 - p + \gamma_S \theta_S  \right)
=
0
\,.
\label{EPEady-circons1}
\end{align}
\begin{corollary}
  Equation (\ref{EPEady-circons1}) and incompressibility imply that
  potential vorticity (PV, denoted as $q$) is conserved along flow
  lines of the fluid velocity $u_S$ in the Eady model,
\begin{align}
\partial_t q + u_S\cdot\nabla q = 0
\quad\hbox{for potential vorticity}\quad
q := \left( 
s {\rm curl}\,u_S
+ 
\nabla \theta_S \times \nabla(u_T + fx) \right)
\cdot \hat{y}
\,.
\label{EPSD-PVthm1}
\end{align}
On denoting $u_S=(u,w)$, $u_T=v$, this potential vorticity may be written as
\[
q = -\pp{\bar{\theta}}{y}\left(\pp{w}{x}+\pp{u}{z}\right) +
\frac{\partial(v+fx,\theta')}{\partial(x,z)}.
\]
\end{corollary}
\noindent Applying the Legendre transform to the
Lagrangian (\ref{Eady-Lag}) yields the energy
\begin{equation}
\label{Eady-energy}
h[u_S, u_T,D,\theta,p] = 
\int_{\Omega} \frac{D}{2}\left(|u_S|^2 + u_T^2\right)
-\frac{g}{\theta_0} D\left(z-\frac{H}{2}\right)\theta_S
\diff{V}.
\end{equation}
\begin{corollary}
The energy \eqref{Eady-energy} is conserved for the 
Eady Boussinesq slice model.
\end{corollary}

\section{Lagrangian-averaged Boussinesq model}
\label{alpha model}
Numerical forecast models are restricted in grid resolution due to the
stringent time requirements of operational forecasting, and hence it
is necessary to perform some form of averaging on the equations in
order to prevent energy and enstrophy accumulating at the gridscale,
either explicitly by introducing extra terms (\emph{i.e.} eddy
viscosities, or Large Eddy Simulation), or implicitly by numerical
stabilisation in advection schemes. All of these examples amount to
some form of Eulerian averaging that leads to dissipation, which is
thought to be detrimental to evolution of fronts. To avoid this,
\cite{Cu2007} suggested that some form of Lagrangian averaging may be
required, also suggesting that it is important for averaging to retain
energy and potential vorticity conservation if agreement with the
SG limiting solution is to be obtained. 

In this section we obtain a Lagrangian averaged Boussinesq model from
a variational principle, and so energy and potential vorticity
conservation will follow immediately. Here, we shall interpret
Lagrangian averaging as a regularisation of the equations that is
consistent with the Lagrangian flow map for slice models in Equation
\eqref{eq:slice map}. This regularisation is obtained by
replacing Equation \eqref{Eady-Lag} with
\begin{align}
\begin{split}
l[u_S, u_T,D,\theta,p] = 
\int_{\Omega} \Bigg[
\frac{D}{2}\left(|u_S|^2 + \alpha^2|\nabla u_s|^2 + u_T^2
+ \alpha^2|\nabla u_T|^2\right)
+D f u_T x
&\\
+\frac{g}{\theta_0} D\left(z-\frac{H}{2}\right)\theta_S
+ p(1-D)
\Bigg] \diff{V},
&
\end{split}
\label{Eady-Alpha-Lag}
\end{align}
where $\alpha$ is a regularisation lengthscale. 
We
obtain the following variational derivatives of this Lagrangian,
\begin{align}
\begin{split}
  \tilde{u}_S & = \frac1D\dede{l}{u_S} = (1-\alpha^2\nabla^2_D)u_S \,, \qquad
  \tilde{u}_T = \frac1D\dede{l}{u_T} = (1-\alpha^2\nabla^2_D)u_T + fx
  \,, \\
  \pi &= \dede{l}{D} = \frac{1}{2}\left(|u_S|^2 + u_T^2\right) + fu_Tx
  - p + \frac{g}{\theta_0}\theta_S\left(z-\frac{H}{2}\right)
  , \\
  \gamma_S & = \frac1D\dede{l}{\theta_S} =
  \frac{g}{\theta_0}\left(z-\frac{H}{2}\right) , \qquad \dede{l}{p} =
  1-D\,,
\end{split}
\label{Eady-Lag-alpha-vars}
\end{align}
where 
\[
\nabla^2_D = \frac{1}{D}\nabla\cdot D\nabla.
\]
Substitution into the Euler-Poincar\'e equations and applying $D=1$ 
gives
\begin{align}
\begin{split}
\partial_t \tilde{u}_S + u_S\cdot\nabla \tilde{u}_S + \nabla u_S^T\tilde{u}_S
-fu_T{\hat{x}}
&= -\nabla p 
 + \frac{g}{\theta_0}\theta_S\hat{z}, \\
\partial_t \tilde{u}_T + u_S\cdot\nabla \tilde{u}_T 
+ fu_S\cdot\hat{{x}}
& =   -\frac{g}{\theta_0}\left(z-\frac{H}{2}\right){s}, \\
\nabla\cdot u_S & =  0, \\
\partial_t \theta_S + u_S\cdot\nabla\theta_S + u_T{s} & =  0, \\
  \tilde{u}_S & =  (1-\alpha^2\nabla^2)u_S,\\
  \tilde{u}_T  &= (1-\alpha^2\nabla^2)u_T .
\end{split}
\label{alpha EP}
\end{align}
This is the Lagrangian averaged Boussinesq Eady slice model.
\begin{corollary}
Equations \eqref{alpha EP} have conserved energy
\[
h = \int_{\Omega} \frac{D}{2}\left(|u_S|^2 + \alpha^2|\nabla u_s|^2 + u_T^2
+ \alpha^2|\nabla u_T|^2\right)
-\frac{g}{\theta_0} D\left(z-\frac{H}{2}\right)\theta_S
 \diff{V}.
\]
\end{corollary}
\begin{corollary}
  Equations \eqref{alpha EP} have Lagrangian potential vorticity
  conservation
\begin{align}
\partial_t q + u_S\cdot\nabla q 
= 0
\quad\hbox{for potential vorticity}\quad
q := \frac{1}{D}\left( s\,
{\rm curl}\,\tilde{u}_S 
+ 
 \nabla \theta_S \times \nabla \tilde{u}_T\right) \cdot \hat{y}
\,.
\end{align}
\end{corollary}

%%%%%%%%%%%%%%%%%%%%%%%%%%%%%%%%
\section{Sliced Compressible Model (SCM)} 
\label{compressible}
In this section we present a model that is a compressible extension of
the Boussinesq Eady model described in the previous section. The aim
of the model is to provide a framework where nonhydrostatic
compressible dynamical cores can be benchmarked in a slice geometry.
Due to the nonlinear equation of state, it is not possible to write
down a compressible slice model with solutions that correspond to
solutions of the full three dimensional equations, and we need to
proceed by replacing the full potential temperature $\theta$ in the
internal energy by the slice component $\theta_S$. This approximation
would be valid if the potential temperature were slowly varying
in the $y$-direction. We derive a
model that has conserved energy, potential vorticity, and supports
baroclinic instability leading to front formation, so that dynamical
cores in this configuration can be compared with the corresponding
model in the SG limit.

In the present notation, the Lagrangian for the Sliced Compressible
Model (SCM) in Eulerian $(x,y,z)$ coordinates is,
\begin{align}
l\big[u_S, u_T,D,\theta_S\big] = 
\int_{\Omega} \frac{D}{2}\left(|u_S|^2 + u_T^2\right)
+ fDu_T x
+
gDz
- Dc_v\theta_S\Pi \diff{V},
\label{SCM-Lag}
\end{align}
where $\Pi$ is the Exner function given by
\[
\Pi = \left(
\frac{p}{p_0}
\right)^{R/c_p},
\]
where $p_0$ is a reference pressure level and $c_p$ and $R$ are gas
constants. The equation for an ideal gas becomes
\[
p_0\Pi^{c_p/R} = D R\theta_S\Pi,
\]
and differentiating with respect to $\theta_S$ and $D$ gives
\begin{eqnarray*}
\pp{\Pi}{\theta_S} & = & \pp{}{\theta_S}
\left(\frac{DR\theta_S}{p_0}\right)^{\gamma-1}, \\
& = & (\gamma-1)\frac{DR\theta_S}{p_0}
\left(\frac{DR\theta_S}{p_0}\right)^{\gamma-2}, \\
& = & \frac{\gamma-1}{\theta_S}\Pi
 =  \frac{c_P-c_v}{c_v\theta_S}\Pi
 = \frac{R}{c_v\theta_S}\Pi.
\end{eqnarray*}
Similarly we obtain 
\[
\pp{\Pi}{D} = \frac{R}{c_vD}{\Pi}.
\]
Note that we use $\theta_S$ in both the internal energy term in the
Lagrangian, and in the equation of state. This removes all
$y$-dependence from the Lagrangian, making a slice model possible.

We obtain the following variational derivatives of this Lagrangian,
\begin{align}
\begin{split}
v_S & = \frac1D\dede{l}{u_S}  =  u_S
\,, \qquad
v_T = \frac1D\dede{l}{u_T}  =  u_T + fx
\,, \\
\dede{l}{D} & =  \frac{1}{2}\left(|u_S|^2 + u_T^2\right) + fu_Tx + gz -
c_p\Pi\theta_S \,, \\
\gamma_S & = \frac1D\dede{l}{\theta_S}  
= \frac1D\dede{l}{\theta_S} = -\, c_p\Pi
\,,
\end{split}
\label{SCM-Lag-vars}
\end{align}
where we have used the decomposition (\ref{theta-decomp}) in the last line.

Substitution of the variational derivatives (\ref{SCM-Lag-vars}) of
the SCM Lagrangian (\ref{SCM-Lag}) into the Euler-Poincar\'e equations
in (\ref{EPSD-system}) gives the system
 \begin{align}
\begin{split}
\left(\partial_t + \mathcal{L}_{u_S}\right)\left(u_S \cdot \dx\right)
 &= -c_p\theta_S\, d \Pi + \,\diff \left(\frac{1}{2}|u_S|^2  - gz \right)
+ fu_T \diff x 
\,, \\
\left(\partial_t + \mathcal{L}_{u_S}\right)(u_T + fx)
& =   \,s c_p\Pi
\,,\\
\left(\partial_t + \mathcal{L}_{u_S}\right) \theta_S  
&= -\,s u_T 
\,,\\
(\partial_t + \mathcal{L}_{u_S}) ({D} \diff S) & =  0
\,.
\end{split}
\label{EPSD-geom-scm}
\end{align}
Consequently, we recover the expected Kelvin circulation conservation law (\ref{EPSD-circ3}) for the SCM in the form
\begin{align}
\frac{d}{dt}\oint_{c(u_S)}\hspace{-2mm} 
\left(u_S  - s^{-1} (u_T+fx) \nabla \theta_S  \right)\cdot \dx
=
0
\,.
\label{EPEady-circons}
\end{align}
\begin{corollary}\label{SCM-PV}
  The system of SCM equations in (\ref{EPSD-geom-scm}) implies that
  potential vorticity $q$ is conserved along flow lines of the fluid
  velocity $u_S$,
\begin{align}
\partial_t q + u_S\cdot\nabla q = 0
\quad\hbox{with potential vorticity}\quad
q := \frac{1}{D}\left( 
s\,{\rm curl}\,u_S
+ 
 \nabla \theta \times \nabla(u_T + fx) \right) \cdot \hat{y}
\,.
\label{EPSD-vorticitythm}
\end{align}
\end{corollary}
\begin{corollary}
These equations are Hamiltonian, with conserved energy
\[
E = \int_{\Omega} \frac{D}{2}\left(|u_S|^2 + u_T^2\right)
-
gDz + c_vD\Pi\theta'\diff{V}.
\]
\end{corollary}
\begin{remark}
The system of SCM equations in (\ref{EPSD-geom3}) may also be written equivalently in \emph{standard} fluid dynamics notation as
\begin{align}
\begin{split}
\partial_t u_S + u_S\cdot\nabla u_S - fu_T\hat{x}
&= -\,c_p\theta\nabla \Pi 
 - g \hat{z}, \\
\partial_t u_T + u_S\cdot\nabla u_T + fu_S\cdot\hat{x} & =   s c_p\Pi, \\
\partial_t\theta_S+ u_S\cdot\nabla \theta_S   
& =   -\, s\,u_T
\,, \\
\frac{\partial{D}}{\partial t} + \nabla  \cdot ( {D} u_S ) & =  0
\,.\end{split}
\label{Eady-EPSDeqns2}
\end{align}
\end{remark}

Next we check that the basic state of these equations supports a shear
profile (and hence allows baroclinic instability and frontogenesis). 
Reverting to more standard notation $u_s=(u,w)$, $u_T=v$, the balance
equations are
\begin{eqnarray}
-fv &=& -c_p\theta'\pp{\pi}{x}, \label{vbal} \\
fu & = & \pp{\bar{\theta}}{y}c_p\Pi, \label{geobal} \\
0 & = & -c_p\theta\pp{\pi}{z} - g. \label{hydros}
\end{eqnarray}
Assuming a $x$-independent temperature field, then Equation
\eqref{vbal} implies that $v=0$. For positive $\theta$, equation
\eqref{hydros} implies that $\Pi$ will increase with height, 
and equation \eqref{geobal} then implies that $u$ decreases
with height, leading to a shear profile in the basic state.

We now compare our SCM with the slice compressible model in
\cite{Cu2008} and identify the differences.  On defining velocity
$\mathbf{u}= (u_S,u_T)$ with $u_S$ in the vertical slice, and $u_T$
transverse to it, the model in \cite{Cu2008} in Eulerian $(x,y,z)$
coordinates becomes, in the present notation,
\begin{align}
\begin{split}
\partial_t u_S + u_S\cdot\nabla u_S 
 - fu_T\hat{x} &= -\,c_p\theta\, \nabla \Pi  - g \hat{z}
\,, \\
\partial_t u_T + u_S\cdot\nabla u_T 
+ fu_S\cdot\hat{x} & =  -\,c_p\theta\, \Pi\,'_0
\,, \\
\partial_t\theta_S+ u_S\cdot\nabla \theta_S   
&  = -\, s\,u_T
\,, \\
\frac{\partial{D}}{\partial t} + \nabla  \cdot ( {D} u_S ) & =  0
\,.
\end{split}
\label{MCUM-geom}
\end{align}

Writing the \cite{Cu2008} equations in Lie-derivative form yields,
  cf. equation (\ref{EPSD-geom3}),
\begin{align}
\begin{split}
(\partial_t + \mathcal{L}_{u_S}) (u_S\cdot \diff x) 
&= -\,c_p\theta\, d \Pi 
 + \diff \left( \frac12 |u_S|^2 - g z \right) 
 + fu_T\diff x
\,, \\
(\partial_t + \mathcal{L}_{u_S}) (u_T + fx)  
& =  -\,c_p\theta\, \Pi\,'_0
\,, \\
(\partial_t + \mathcal{L}_{u_S}) \theta   
& =   -\, s\,u_T
\,, \\
(\partial_t + \mathcal{L}_{u_S}) ({D} \diff S) & =  0
\,.
\end{split}
\label{MCUM-geomeqns}
\end{align}
These equations differ from the SCM equations in (\ref{EPSD-geom3}),
by only one term. Namely, the right hand sides of the second equation
in each set differ, with $(-c_p \Pi\,'_0\,\theta)$ in these equations
and $(s c_p\Pi)$ in (\ref{EPSD-geom3}).  It turns out that this
single difference has important consequences for their respective
circulation laws.\\

The circulation law for  the compressible slice models in \cite{Cu2008} is similar to that for the SCM in the previous section, but with one important difference. Namely,

\begin{theorem}\rm
Circulation for the compressible slice models in \cite{Cu2008} is not conserved. Instead, we find
\begin{eqnarray}
\frac{d}{dt} 
\oint_{c(u_S)} 
\Big(u_S-s^{-1}(u_T+fx)  \nabla \theta  \Big)
\cdot \diff x = 
- \oint_{c(u_S)} 
c_p\theta\, \nabla \Pi \cdot \diff x
\,.
\label{circ-de-slice}
\end{eqnarray}
\end{theorem}
\begin{proof}
The proof uses the first three equations in the system (\ref{MCUM-geomeqns}). The middle two equations yield
\begin{eqnarray*}
(\partial_t + \mathcal{L}_{u_S}) (-s^{-1}(u_T+fx)  \diff \theta) 
=
\frac12  \diff \, \big(s^{-1}c_p \Pi\,'_0 \theta^2 + u_T^2\big) 
+ fx  \diff u_T
.\end{eqnarray*}
Combining this formula with the first equation in the system (\ref{MCUM-geomeqns}) then yields the circulation law,  (\ref{circ-de-slice}).
\end{proof}

\begin{corollary}\label{PV-Cu2008}\rm
Equation (\ref{circ-de-slice}) implies that potential vorticity (PV, still denoted as $q$) is \emph{created}  along flow lines of the fluid velocity $u_S$, as 
\begin{align}
\partial_t q + u_S\cdot\nabla q 
= 
c_pD^{-1} \nabla \Pi\times \nabla \theta \cdot \hat{y}
\quad\hbox{with PV given by}\quad
q := 
D^{-1}\left( 
s\,{\rm curl}\,u_S
+ 
 \nabla \theta \times \nabla (u_T+fx) 
\right) \cdot \hat{y}
\,.
\label{slice-PVeqn}
\end{align}
\end{corollary}
\begin{proof}
Applying Stokes theorem to the circulation equation in (\ref{circ-de-slice}) yields
\begin{align}
\frac{d}{dt} 
\int\!\!\!\int_{\partial S(u_S)} 
\Big( 
{\rm curl}\,u_S
+ 
s^{-1} \nabla \theta \times \nabla (u_T+fx) 
\Big)  
\cdot \hat{y}\diff S = 
\int\!\!\!\int_{\partial S(u_S)} 
c_p \nabla \Pi\times \nabla \theta \cdot \hat{y}\diff S
\,,
\label{Stokes-slice}
\end{align}
where $\hat{y}\diff S$ is the surface element in the vertical slice, whose normal vector is $\hat{y}$. 
Expanding the time derivative in (\ref{Stokes-slice}) and applying the Lie derivative relation for $D$ in the last equation of the system (\ref{MCUM-geomeqns}), which is the continuity equation for $D$, then yields the local PV evolution equation in (\ref{slice-PVeqn}).
\end{proof}

\begin{remark}\rm
  This is the main difference between the SCM here and in
  \cite{Cu2008}. According to Corollary \ref{SCM-PV}, the potential
  vorticity in the SCM is conserved and this conservation is a general
  property of this class of Euler-Poincar\'e equations, as given by
  Corollary \ref{EPSCM-PVthm}. In contrast, according to Corollary
  \ref{PV-Cu2008}, the potential vorticity in the model of
  \cite{Cu2008} (when viewed as a slice model) is \emph{created}
  whenever the gradients of $\theta$ and $\Pi$ are not aligned. This,
  combined with the lack of a conserved energy meant that it was not
  possible to obtain long time asymptotic convergence results in a
  compressible model, because these quantities \emph{are} conserved in
  the equivalent SG model; \cite{Cu2008} restricted to looking at the
  asymptotic magnitude of the geostrophic imbalance in the solution.
  Our new model addresses this problem, allowing asymptotic limit
  tests to be performed with compressible models in a slice
  configuration.
\end{remark}

\section{Summary and Outlook}
\label{summary}
In this paper we have shown how to construct variational models for
geophysical fluid dynamics problems in a vertical slice configuration
in which there is motion transverse to the slice, but the velocity
field is independent of the transverse coordinate. (The vertical slice
configuration may be taken as the $x$-$z$ plane. Then the transverse
coordinate is $y$.) Any model developed in this framework has a
conserved energy, and corresponding conserved potential vorticity. The
formulation has a number of interesting geometric features, arising
from the semidirect product structure of the slice subgroup of the
group of three-dimensional diffeomorphisms. Firstly, the formulation
leads to a Kelvin-Noether circulation theorem in which circulation is
preserved on arbitrary loops in the slice, unlike the usual
circulation theorem in which circulation is only preserved on
isentropic surfaces. Secondly, as shown in Appendix \ref{HamForm}, the
equations can always be rewritten in terms of a pair of two
dimensional momenta, one comprising the $x$- and $z$-components of
linear momentum, and one formed from the temperature and the
$y$-component of linear momentum, plus the density. This formulation
involving only two-dimensional momenta and density means that
potential vorticity conserving numerical schemes for the shallow-water
equations can be adapted for vertical slice problems. In the
shallow-water case, the equations can be written in the form
\[
\left(\pp{}{t}+\mathcal{L}_u\right)v\cdot \diff x + \diff \pi = 0,
\]
where $u$ is the velocity, $v$ is the total momentum divided by the
layer thickness, and $\pi$ is a pressure. It is possible using
mimetic/discrete exterior calculus methods \cite{ThCo2012} to use $u$
as a prognostic variable, but to also apply $\diff\,$ to the above
equation, use some chosen stable conservative advection scheme for
potential vorticity, and then to obtain a discrete form of
$\mathcal{L}_uv\cdot\diff x$ which is consistent with that scheme (so
that potential vorticity advection is stabilised even though it is a
diagnostic variable).  This programme cannot be easily extended to
three dimensions when advected temperature is present, since we gain
an extra term of the form $G\diff\theta$ (for some scalar function
$G$), and it is not currently clear how to obtain a discrete form of
$\mathcal{L}_uv$ that is consistent with a stable advection scheme for
the Ertel potential vorticity. However, for the slice model equation
set \eqref{EPSD-geom-form}, it should be possible to use conservative
advection schemes for the last three equations, then apply $\diff$ to
equation \eqref{EPSD-geom3}, apply a stable conservative advection
scheme for potential vorticity, and obtain a discrete form of
$\mathcal{L}_{u_s} v_s\cdot\diff x$ that is consistent with that
scheme. This becomes possible for slice models, the extra term can be
moved inside the Lie derivative to obtain equation \eqref{EPSD-geom3}.

This work has led to the development of new model equations: a
Lagrangian-averaged form of the Eady model of frontogenesis and a new
compressible model. We plan to use both of these models to investigate
how to improve prediction of front evolution, following the programme
set out in \cite{Cu2007}. Whilst solutions of the slice compressible
model do not recover solutions of the full three dimensional
equations, this model approximates the slice Boussinesq model in the
Boussinesq limit, and is easily obtained by very minor modifications
to standard dynamical core slice configurations, so will allow
asymptotic limit analysis to be performed with compressible codes,
addressing a problem highlighted in \cite{Cu2008}.

\subsection*{Acknowledgments}
The authors are grateful to Mike Cullen and Abeed Visram for very
useful and interesting discussions about slice models. The work by DDH
was partially supported by Advanced Grant 267382 from the European
Research Council and the Royal Society of London Wolfson Award
Scheme. The work by CJC was partially supported by the Natural
Environment Research Council Next Generation Weather and Climate
programme.

\bibliography{slice_variational}

\appendix

\section*{Appendices}

\section{Euler--Poincar\'e semidirect-product formulation}\label{EPForm}
The advection equations (\ref{D-eqn})--(\ref{theta-eqn}) for $(\theta_S,{s} )$ and $D$ may be rewritten in Lie-derivative notation as 
\begin{align}
\begin{split}
\partial_t \left(\theta_S,{s} \right)
 & = - \mathcal{L}_{(u_S,u_T)} \left(\theta_S,{s} \right)
 =  \left(-\,u_S\cdot\nabla\theta_S - u_T{s},0\right)
, \\
\partial_t (D\, \diff S) & = - \mathcal{L}_{u_S} (D\,\diff S)=  -\,{\rm div}(u_SD)\,\diff S
\,.
\end{split}
\label{advection-GM}
\end{align}
The corresponding infinitesimal variations in $(\theta_S,{s} )$ and $D$, in (\ref{var-formula}) induced by the Lie-derivative actions of the Lie algebra of vector fields $\mathfrak{X}(\Omega)\circledS\mathcal{F}(\Omega)$ are given by:
\begin{align}
\begin{split}
\delta\left(\theta_S,{s} \right)
 & = - \mathcal{L}_{(w_S,w_T)} \left(\theta_S,{s} \right)
 =  \left(-w_S\cdot\nabla\theta_S - w_T{s},0\right)
, \\
\delta D\, \diff S & = - \mathcal{L}_{w_S} (D\,\diff S)=  -\,{\rm div}(w_SD)\,\diff S
\,.
\end{split}
\label{var-formula-GM1}
\end{align}
The infinitesimal variations in $(u_S,u_T)$ in (\ref{var-formula}) may be expressed in terms of the adjoint action in the Lie algebra $\mathfrak{X}(\Omega)\circledS\mathcal{F}(\Omega)$ of the semidirect-product group $\Diff(\Omega)\circledS \mathcal{F}(\Omega)$. Namely,
\begin{align}
\begin{split}
\delta \left(u_S,u_T\right) 
& =  \left(\partial_t{w}_S,\partial_t{w}_T\right) - {\rm ad}_{(u_S,u_T)}\left(w_S,w_T\right)
\\ 
&= 
 \left(\partial_t{w}_S + [u_S,w_S], \, \partial_t{w}_T + u_S\cdot\nabla w_T - w_S\cdot\nabla u_T\right)
\,.
\end{split}
\label{var-formula-GM2}
\end{align}
For a Lagrangian functional $l[(u_S,u_T),(\theta_S,{s} ),D]: 
(\mathfrak{X}\circledS\mathcal{F}(\Omega))\circledS ( (\Lambda^0(\Omega) \times \mathbb{R}) \times \Lambda^2(\Omega))\to\mathbb{R} $, one defines Hamilton's 
principle using the $L^2$ pairing, which is denoted as $\langle\,\cdot\,,\,\cdot\, \rangle$. Hence, 
inserting the infinitesimal variational formulas in (\ref{var-formula}) for $(u_S,u_T)$, $(\theta_S,{s} )$ and $D$ yields, in semidirect-product notation, 
\begin{align}
\begin{split}
0 
&= \delta S \\
&= \delta \int_0^T l[(u_S,u_T),(\theta_S,{s} ),D] \diff{t} \\
&= \int_0^T
\left\langle 
\dede{l}{(u_S,u_T)},\delta (u_S,u_T) 
\right\rangle
+
\left\langle 
\dede{l}{(\theta_S,{s} )},\delta (\theta_S,{s} ) 
\right\rangle
+ \left\langle \dede{l}{D}, \delta D \right\rangle
\diff{t} \\
& =  \int_0^T 
\left\langle 
\dede{l}{(u_S,u_T)}, \pp{}{t}(w_S,w_T) - {\rm ad}_{(u_S,u_T)} (w_S,w_T)
\right\rangle
\\
&  \qquad  
+ \left\langle \dede{l}{(\theta_S,{s} )},\,
 - \mathcal{L}_{(w_S,w_T)} \left(\theta_S,{s} \right)
\right\rangle  
+ \left\langle \dede{l}{D},\, -\,{\rm div}(w_SD) \right\rangle 
\diff{t}
\\
& =  \int_0^T \left\langle -\pp{}{t}\dede{l}{(u_S,u_T)} -  {\rm ad}^*_{(u_S,u_T)}\dede{l}{(u_S,u_T)} 
+ 
\dede{l}{(\theta_S,{s} )}\diamond (\theta_S,{s} )
+
\left(\dede{l}{D}\diamond D,\,0\right)
,\,\left(w_S,w_T\right)\right\rangle  \diff{t} \\
& \qquad 
+ \left[
\left\langle \dede{l}{(u_S,u_T)},\,\left(w_S,w_T\right) \right\rangle
\right]_0^T
.
\end{split}
\label{HP-sdpnotation}
\end{align}
In comparison, see equation (\ref{HP-vectornotation}) for the \emph{same} Hamilton's principle in vector notation.
As before, the last term in the previous equation vanishes because $(w_S,w_T)$ vanishes at the endpoints. The ${\rm ad}^*$ notation in (\ref{HP-sdpnotation}) denotes the dual of the ad operation with respect to the $L^2$ pairing $\langle\,\cdot\,,\,\cdot\, \rangle$ \cite{HoMaRa1998}. Explicitly, the $L^2$ dual of the ad operation is defined by

\begin{align}
\left\langle 
{\rm ad}^*_{(u_S,u_T)}\dede{l}{(u_S,u_T)}, \left(w_S,w_T\right)
\right\rangle
=
\left\langle 
\dede{l}{(u_S,u_T)}, {\rm ad}_{(u_S,u_T)}\left(w_S,w_T\right)
\right\rangle
.
\label{EPSD-adstar}
\end{align}
Likewise, the diamond $(\diamond)$ operation is defined in the present notation by the $L^2$ pairings,
\begin{align}
\begin{split}
\left\langle
\dede{l}{(\theta_S,{s} )}\diamond (\theta_S,{s} ) ,\,\left(w_S,w_T\right)
\right\rangle
&:=
\left\langle \dede{l}{(\theta_S,{s} )},\,
 - \mathcal{L}_{(w_S,w_T)} \left(\theta_S,{s} \right)
\right\rangle
,
\\
\left\langle
\left(\dede{l}{D}\diamond D,\,0\right)
,\,\left(w_S,w_T\right)\right\rangle 
&:=
\left\langle \dede{l}{D},\, -\,\mathcal{L}_{w_S} D \right\rangle
=
\left\langle \dede{l}{D},\, -\,{\rm div}(w_SD) \right\rangle
.
\end{split}
\label{EPSD-diamond}
\end{align}
Hence, the last equality of (\ref{HP-sdpnotation}) yields the Euler-Poincar\'e equations on the dual Lie algebra
$(\mathfrak{X}(\Omega)\circledS\mathcal{F}(\Omega))^*$ 
with the advected areal density $D\in\Lambda^2$ and advected scalars $(\theta_S,{s} )\in\Lambda^0\times  \mathbb{R}$ in semidirect-product form, as
\begin{align}
\pp{}{t}\dede{l}{(u_S,u_T)} + {\rm ad}^*_{(u_S,u_T)}\dede{l}{(u_S,u_T)} 
= 
\dede{l}{(\theta_S,{s} )}\diamond (\theta_S,{s} )
+
\left(\dede{l}{D}\diamond D,\,0\right)
.
\label{EPSD-systemGM}
\end{align}
The system (\ref{EPSD-systemGM}) is completed by including the advection equations (\ref{advection-GM})
for $D$ and $(\theta_S,{s} )$.

\section{Lie-Poisson Hamiltonian formulation}\label{HamForm}

\subsection{Equations on the dual of  
$(\mathfrak{X}\circledS\mathcal{F}(\Omega))\circledS ( (\Lambda^0(\Omega) \times \mathbb{R}) \times \Lambda^2(\Omega))$}

The Legendre transformation to the Hamiltonian is defined by, 
\begin{align}
h[(m_S,\,m_T),(\theta_S,{s} ),D]
= 
\left\langle (m_S,\,m_T), (u_S,\,u_T) \right\rangle - l[(u_S,\,u_T),(\theta_S,{s} ),D]
\,.
\label{EPSD-Leg}
\end{align}
Therefore, we find the variational relations
\begin{align}
(m_S,\,m_T) =\dede{l}{(u_S,u_T)}
\,,\quad
(u_S,\,u_T) =\dede{h}{(m_S,m_T)}
\,,\quad
\dede{h}{(\theta_S,{s} )}
=
-\,\dede{l}{(\theta_S,{s} )}
\,,\quad
\dede{h}{D}
=
-\,\dede{l}{D}
\,.
\end{align}
Consequently, the system (\ref{EPSD-systemGM}) may be written in terms of the Hamiltonian as
\begin{align}
\pp{}{t}(m_S,m_T) =
-\, {\rm ad}^*_{\delta{h}/\delta(m_S,m_T)}(m_S,m_T)
-\,\dede{h}{(\theta_S,{s} )}\diamond (\theta_S,{s} )
-\,(\dede{h}{D}\diamond D,\,0)
\,.
\label{EPSD-systemHam}
\end{align}
The advection equations (\ref{advection-GM}) for $(\theta_S,{s} )$ and $D$  are then written as
\begin{align}
\begin{split}
\pp{}{t}(\theta_S,{s} )
 & = - \mathcal{L}_{\delta{h}/\delta(m_S,m_T)} \left(\theta_S,{s} \right)
, \\
\pp{}{t} (D,\,0) & = - \mathcal{L}_{\delta{h}/\delta(m_S,m_T)} (D,\,0)
\,.
\end{split}
\label{advection-GM-again}
\end{align}
Hence, the entire system (\ref{EPSD-systemHam})--(\ref{advection-GM-again}) may be written in Hamiltonian form as
\begin{align}
\pp{}{t}
\begin{bmatrix}
(m_S,m_T) \\ (\theta_S,{s} ) \\ (D,\,0)
\end{bmatrix}
=
-\,
\begin{bmatrix}
{\rm ad}^*_{\Box} (m_S,m_T) &  \Box \diamond (\theta_S,{s} )  & \Box \diamond  (D,\,0)
\\ 
 \mathcal{L}_{\Box} (\theta_S,{s} ) &  0  &  0 
 \\
 \mathcal{L}_{\Box} (D,\,0) &  0  &  0 
\end{bmatrix}
\begin{bmatrix}
\delta{h}/\delta(m_S,m_T) \\ \delta{h}/\delta(\theta_S,{s} ) \\ \delta{h}/\delta(D,\,0)
\end{bmatrix}
,
\label{Ham-GM}
\end{align}
in which the box $\Box$ indicates the appropriate substitutions. The matrix operator in (\ref{Ham-GM}) defines  
a Lie-Poisson bracket dual to the semidirect product action 
$(\mathfrak{X}\circledS\mathcal{F}(\Omega))\circledS ( (\Lambda^0(\Omega) \times \mathbb{R}) \times \Lambda^2(\Omega))$ with coordinates $(u_S,u_T)\in \mathfrak{X}\circledS\mathcal{F}(\Omega)$,
$(\theta_S,{s} )\in\Lambda^0(\Omega) \times \mathbb{R}$ and $D\in \Lambda^2(\Omega)$.
This identification of the Lie-Poisson bracket with the dual of a Lie algebra action guarantees that it satisfies the Jacobi identity.  Explicitly, the Lie-Poisson bracket is the following
\begin{align}
\Big\{ f,\, h\,\Big\}
=
-\,
\left\langle
\begin{bmatrix}
\delta{f}/\delta(m_S,m_T) \\ \delta{f}/\delta(\theta_S,{s} ) \\ \delta{f}/\delta(D,\,0)
\end{bmatrix}^T,
\begin{bmatrix}
{\rm ad}^*_{\Box} (m_S,m_T) &  \Box \diamond (\theta_S,{s} )  & \Box \diamond  (D,\,0)
\\ 
 \mathcal{L}_{\Box} (\theta_S,{s} ) &  0  &  0 
 \\
 \mathcal{L}_{\Box} (D,\,0) &  0  &  0 
\end{bmatrix}
\begin{bmatrix}
\delta{h}/\delta(m_S,m_T) \\ \delta{h}/\delta(\theta_S,{s} ) \\ \delta{h}/\delta(D,\,0)
\end{bmatrix}
\right\rangle
,
\label{LPB-GM1}
\end{align}
where $\langle\,\cdot\,,\,\cdot\, \rangle$ denotes the $L^2$ pairing. 

Expanding out the operations in (\ref{LPB-GM1}) makes it clear that this Lie-Poisson bracket has the required property of being antisymmetric under exchange of $f$ and $h$. That is, 
$\{ h,\, f\}=-\{ f,\, h\}$, which is evident upon expanding out the operations to express the bracket in (\ref{LPB-GM1}) equivalently as
\begin{align}
\begin{split}
\Big\{ f,\, h\,\Big\}
& =
-\,
\left\langle
(m_S,m_T) ,\,
\left[\frac{\delta{f}}{\delta(m_S,m_T)},\, \frac{\delta{h}}{\delta(m_S,m_T)}\right]
\right\rangle
\\& \quad + 
\left\langle
(\theta_S,{s} ),\, 
\mathcal{L}^+_{{\delta{f}}/{\delta(m_S,m_T)}}{\frac{\delta{h}}{\delta(\theta_S,{s} )}}
-
\mathcal{L}^+_{{\delta{h}}/{\delta(m_S,m_T)}}{\frac{\delta{f}}{\delta(\theta_S,{s} )}}
\right\rangle
\\& \quad + 
\left\langle
(D,\,0),\, 
\mathcal{L}^+_{{\delta{f}}/{\delta(m_S,m_T)}}{\frac{\delta{h}}{\delta(D,\,0)}}
-
\mathcal{L}^+_{{\delta{h}}/{\delta(m_S,m_T)}}{\frac{\delta{f}}{\delta(D,\,0)}}
\right\rangle
.
\end{split}
\label{LPB-GM2}
\end{align}
Here, $\mathcal{L}^+$ denotes the $L^2$ adjoint of the Lie derivative $\mathcal{L}$. In particular, 
upon denoting ${\delta{f}}/{\delta(m_S,m_T)}=(w_S,w_T)$, we find the following relations among the operations 
$\mathcal{L}^+$,
 $\mathcal{L}$  and $\diamond$,
\begin{align}
\left\langle 
 \left(\theta_S,{s} \right),\,
\mathcal{L}^+_{(w_S,w_T)}
\dede{h}{(\theta_S,{s} )}
\right\rangle
:=
\left\langle \dede{h}{(\theta_S,{s} )},\,
\mathcal{L}_{(w_S,w_T)} \left(\theta_S,{s} \right)
\right\rangle
=:
\left\langle
-\,\dede{h}{(\theta_S,{s} )}\diamond (\theta_S,{s} ) ,\,\left(w_S,w_T\right)
\right\rangle
\,.
\label{Lie-transpose}
\end{align}

\begin{remark}\rm 
If desired, one may now substitute the expressions for Lie derivative (\ref{advection-GM}), ${\rm ad}^*$ (\ref{EPSD-adstar}) and diamond $(\diamond)$ (\ref{EPSD-diamond}) into the $L^2$ pairings (\ref{LPB-GM1}) or (\ref{LPB-GM2}) to find the Lie-Poisson bracket $\{ f,\, h\}$ as an integral over the slice domain, $\Omega$, involving ordinary vector calculus operations. However, the present forms (\ref{LPB-GM1}) and (\ref{LPB-GM2}) readily reveal its semidirect-product nature and suggest further rearrangements, which we pursue next.
\end{remark}

\subsection{Equations on the dual of  $\mathfrak{X}\circledS(\Lambda^0\oplus \Lambda^2\oplus \Lambda^0)$}

To explore the particular case at hand further, one may rewrite the system of equations (\ref{D-eqn}), (\ref{theta-eqn}) and (\ref{EPSD-system}) equivalently as,
\begin{align}
\begin{split}
\pp{}{t}\dede{l}{u_S} 
 &= 
 -\,{\rm ad}^*_{u_S}\dede{l}{u_S}
- \dede{l}{\theta_S}\nabla\theta_S
- \dede{l}{u_T}\nabla u_T
+ D \nabla \dede{l}{D}
\,, \\
\pp{}{t}\dede{l}{u_T} 
& = -\,\mathcal{L}_{u_S}\dede{l}{u_T}  -\dede{l}{\theta_S}{s}
,\\
\pp{}{t}\theta_S &=  - \mathcal{L}_{u_S}\theta_S - u_T{s} 
,\\
\pp{}{t}D &= -\,\mathcal{L}_{u_S}D
\,,
\end{split}
\end{align}
where $\mathcal{L}_{u_S}$ denotes Lie derivative along the vector field $u_S$ and we have identified $\mathcal{L}_{u_S}$ and ${\rm ad}^*_{u_S}$ when acting on the 1-form density $\delta{l}/\delta{u_S}$ in the first equation. For more details in this matter, see \cite{HoMaRa1998}.

We define the Legendre transformation to the Hamiltonian in this case by
\begin{align}
h[m_S,\,m_T\,,\theta_S,D;\,{s} ]
= 
\left\langle m_S,\,u_S \right\rangle 
+ \left\langle m_T,\,u_T \right\rangle
- l[u_S,\,u_T,\,\theta_S,\,D;\,{s} ]
\,,
\label{EPSD-Leg2}
\end{align}
where the semicolon $[\,\dots\,;{s} ]$ denotes parametric dependence on the constant 
${s}\in\mathbb{R}$.
The Legendre transformation (\ref{EPSD-Leg2}) yields the variational relations
\begin{align}
m_S =\dede{l}{u_S}
\,,\quad
u_S =\dede{h}{m_S}
\,,\quad
m_T =\dede{l}{u_T}
\,,\quad
u_T =\dede{h}{m_T}
\,,\quad
\dede{h}{\theta_S}
=
-\,\dede{l}{\theta_S}
\,,\quad
\dede{h}{D}
=
-\,\dede{l}{D}
\,.
\end{align}
Consequently, the system (\ref{EPSD-systemGM}) may be written in terms of the Hamiltonian as
\begin{align}
\begin{split}
\pp{}{t}m_S 
 &= 
 -\,{\rm ad}^*_{\delta h/\delta m_S}m_S
- m_T\nabla \dede{h}{m_T}
+ \dede{h}{\theta_S}\nabla\theta_S
- D \nabla \dede{h}{D}
\,, \\
\pp{}{t}m_T 
& = -\,\mathcal{L}_{\delta h/\delta m_S}m_T  + \dede{h}{\theta_S}{s}
\,,\\
\pp{}{t}\theta_S &=  - \mathcal{L}_{\delta h/\delta m_S}\theta_S - \frac{\delta h}{\delta m_T}{s} 
\,,\\
\pp{}{t}D &= -\,\mathcal{L}_{\delta h/\delta m_S}D
\,,
\end{split}
\end{align}
The corresponding Hamiltonian matrix is
\begin{align}
\pp{}{t}
\begin{bmatrix}
m_S \\ m_T \\ \theta_S \\ D
\end{bmatrix}
=
-\,
\begin{bmatrix}
{\rm ad}^*_{\Box} m_S &  \Box \diamond m_T  &  \Box \diamond \theta_S  & \Box \diamond  D
\\ 
 \mathcal{L}_{\Box} m_T &  0  &  -\,{s}  & 0
 \\
  \mathcal{L}_{\Box} \theta_S &  {s}   &  0 & 0
 \\
 \mathcal{L}_{\Box} D &  0  &  0 & 0 
\end{bmatrix}
\begin{bmatrix}
\delta{h}/\delta m_S \\ \delta{h}/\delta m_T \\ \delta{h}/\delta \theta_S \\ \delta{h}/\delta D
\end{bmatrix}
,
\label{Ham-GM1}
\end{align}
in which the box $\Box$ indicates the appropriate substitutions. 

After this rearrangement, one recognises (\ref{Ham-GM1}) as the Hamiltonian matrix for the Lie-Poisson bracket on the dual of the semidirect-product Lie algebra $\mathfrak{X}\circledS(\Lambda^0\oplus \Lambda^2\oplus \Lambda^0)$ with a \emph{symplectic two-cocycle} between $m_T$ and $\theta_S$. 
The Lie bracket for this semidirect-product algebra is
\begin{align}
\big[(X,f,\omega,g),\,(\tilde{X},\tilde{f},\tilde{\omega},\tilde{g})\big]
=
\big([X,\tilde{X}],\, X(\tilde{f})-\tilde{X}(f),\, X(\tilde{\omega})-\tilde{X}(\omega),\,X(\tilde{g})-\tilde{X}(g) \big)
\,,\label{Liebracket1}
\end{align}
where, \emph{e.g.}, $X(\tilde{f})=\mathcal{L}_X\tilde{f}$ denotes Lie derivative of $\tilde{f}$ by vector field $X$. 
The dual coordinates are: 
$m_S$ dual to $X\in\mathfrak{X}$; $m_T$ to $f\in\Lambda^0$; $\theta_S$ to $\omega\in\Lambda^2$; and $D$ to $g\in\Lambda^0$.
The spaces in which the coordinates themselves are defined are $(m_S,m_T,\theta_S,D)\in(\Lambda^1\otimes \Lambda^2, \Lambda^2,\Lambda^0,\Lambda^2)$ and ${s}\in\mathbb{R}$ is a parameter. The second part of the bracket  (\ref{Ham-GM1}) is the standard two-cocycle (symplectic form) on $\Lambda^0\oplus \Lambda^2$ arising from the natural projection $\mathfrak{X}\circledS(\Lambda^0\oplus \Lambda^2\oplus \Lambda^0)\to \Lambda^0\oplus \Lambda^2$. 

\begin{remark}\rm
The Hamiltonian matrix with the two-cocycle in (\ref{Ham-GM1}) has been seen before. Namely, it is the same as that for ${}^4He$ superfluids \cite{DzVo1980,HoKu1982} in the spatially two-dimensional case. For ${}^4He$ superfluids, the function $\theta_S$ here plays the role of the phase of the Bose-condensate wave function, whose gradient $\nabla\theta_S$ is the superfluid velocity. The other variables $m_S$, $m_T$ and $D$ correspond respectively, to total momentum density, mass density and entropy density of the superfluid.  
\end{remark}

\subsection{Equations on the dual of  $\mathfrak{X}_1\circledS(\mathfrak{X}_2\oplus \Lambda^0)$}\label{X(S)X-append}

\cite{HoKu1982} showed that the two-cycle in (\ref{Ham-GM1}) may be removed by transforming to new variables 
\begin{align}
(m_S,m_T,\theta_S,D) \to (m_S,m_R,D)
\quad\hbox{where}\quad
m_R:= ({s})^{-1}m_T\nabla \theta_S
\,.
\label{change2P}
\end{align}
The quantity $m_R$ is the momentum map for right action of the diffeomorphisms on the buoyancy $\theta_S$ in two spatial dimensions, see \cite{HoMa2004} for more details. 
The resulting Lie-Poisson bracket has the standard form dual to the Lie algebra $\mathfrak{X}_1\circledS(\mathfrak{X}_2\oplus \Lambda^0)$, whose Lie bracket is 
\begin{align}
\begin{split}
\big[(X_1,X_2,f),\,(\tilde{X}_1,\tilde{X}_2,\tilde{f})\big]
&=
\\  
\big([X_1,\tilde{X}_1],\, [X_2,\tilde{X}_2] & + [X_1,\tilde{X}_2]-[\tilde{X}_1,X_2],\, 
X_1(\tilde{f})-\tilde{X}_1(f)\big)
\,.
\end{split}
\label{Liebracket2}
\end{align}
Dual coordinates in this case are: $m_S$ dual to $X_1\in \mathfrak{X}_1$; $m_R$ to $X_2\in \mathfrak{X}_2$;
and $D$ to $f\in \Lambda^0$.

Transformation of the Hamiltonian matrix (\ref{Ham-GM1}) into these variables yields
 the following Lie-Poisson Hamiltonian system 
\begin{equation}
\pp{}{ t}
\begin{bmatrix}
m_S \\ m_R  \\  D
\end{bmatrix}
=
-
\begin{bmatrix}
{\rm ad}^*_{\square} m_S & {\rm ad}^*_{\square} m_R  & \Box\diamond D
\\ 
{\rm ad}^*_{\square} m_R & {\rm ad}^*_{\square} m_R  & 0
\\ 
\mathcal{L}_{\Box}D & 0 & 0 
\end{bmatrix}
\begin{bmatrix}
\delta h/\delta m_S =: u_S
\\ 
\delta h/\delta m_R =: u_R 
\\ 
\delta h/\delta D =: p
\end{bmatrix}
.\label{MRMS-system}
\end{equation}
This system produces a system of equations for relative momentum $(m_S-m_R)$, momentum map $m_R= ({s})^{-1}m_T\nabla \theta_S$ and mass density $D$, given by
\begin{align}
\begin{split}
\partial_t (m_S - m_R) &= -\, {\rm ad}^*_{u_S} (m_S - m_R) - p \diamond D
\,,\\
\partial_t m_R &= -\, {\rm ad}^*_{(u_S+u_R)} m_R
\,,\\
\partial_t D &= -\, \mathcal{L}_{u_S} D
\,.
\end{split}
\label{MR-MS-eqn}
\end{align}
Upon evaluating $p \diamond D = D\nabla p$, the first of these
equations explains the geometric origin of the Kelvin-Noether
circulation theorem (\ref{EPSD-circthm}) that was found by direct
manipulation in Section \ref{KNthm-sec}. Together, the three equations
in (\ref{MR-MS-eqn}) show that the slice dynamics may be expressed in
terms of $(m_S,m_R,D)$ as a Lie-Poisson Hamiltonian system on the
semidirect product
\[
{\rm Diff}_1(\Omega)\,\circledS\, \big({\rm Diff}_2(\Omega)\times \Lambda^2(\Omega)\big)
\,,\]
in the slice domain $\Omega$. When $D=1$ is imposed, we have $\nabla\cdot u_S=0$ and this simplifies to 
\[
{\rm SDiff}_1(\Omega)\,\circledS\,{\rm Diff}_2(\Omega)
\,.\]

\end{document}